\documentclass[12pt,reqno]{article}

\usepackage{amsmath}
\usepackage{amssymb}
\usepackage{amsopn}
\usepackage{amsthm}
\usepackage{amstext}
\def\Ker{\mathrm{Ker}\,}
\def\diag{\mathrm{diag\,}}
\def\A{\mathcal A}
\def\csr{\mathrm{csr}}
\def\U{\mathcal U}
\def\M{\mathrm M}
\def\D{\mathrm D}
\def\Z{\mathbb Z}
\def\C{\mathbb C}
\def\B{\mathcal B}
\def\A{\mathcal A}
\def\E{\mathcal E}
\def\rank{\mathrm{rank\,}}
\def\Times{\!\times}
\def\dcup{\mathop{\displaystyle\cup}\limits}
\def\doplus{\mathop{\displaystyle\bigoplus}\limits}

\newtheorem{lemma}{Lemma}[section]
\newtheorem{theorem}[lemma]{Theorem}

\newtheorem{corollary}[lemma]{Corollary}

\newtheorem{example}[lemma]{Example}

\begin{document}

\pagestyle{myheadings} \markboth{\rm Yifeng Xue}{\rm Classification of $C(M)\times_\theta\Z_p$
for certain pairs $\mathbf{(M,\theta)}$}

\vspace*{2cm}

\centerline{\Large\bf Classification of the crossed product $\mathbf{C(M)\times_\theta\Z_p}$}

\smallskip

\centerline{\Large\bf for certain pairs $\mathbf{(M,\theta)}$}

\bigskip

\centerline{\bf Yifeng Xue}

\thispagestyle{empty} %\baselineskip=14pt

\footnote[0]{Project supported by Natural Science Foundation of China (no.10771069) and Shanghai\par
\ \ Leading Academic Discipline Project(no.B407)}

\begin{abstract}
Let $M$ be a separable compact Hausdorff space with $\dim M\le 2$ and $\theta\colon M\rightarrow M$ be a homeomorphism with
prime period $p$ ($p\ge 2$). Set $M_\theta=\{x\in M\vert\,\theta(x)=x\}\not=\varnothing$ and $M_0=M\backslash M_\theta$.
Suppose that $M_0$ is dense in $M$ and $\mathrm H^2(M_0/\theta,\Z)\cong 0$,
$\mathrm H^2(\chi(M_0/\theta),\Z)\cong 0$. Let $M'$ be another separable compact Hausdorff
space with $\dim M'\le 2$ and $\theta'$ be the self--homeomorphism of $M'$ with prime period $p$. Suppose that
$M_0'=M'\backslash M_{\theta'}'$ is dense in $M'$. Then $C(M)\times_\theta\Z_p\cong
C(M')\times_{\theta'}\Z_p$ iff there is a homeomorphism $F$ from $M/\theta$ onto $M'/\theta'$
such that $F(M_\theta)=M'_{\theta'}$. Thus, if $(M,\theta)$ and $(M',\theta')$ are orbit
equivalent, then $C(M)\times_\theta\Z_p\cong C(M')\times_{\theta'}\Z_p$.
\end{abstract}

\section{Introduction}

\quad\ The classification of dynamical systems and corresponding $C^*$--crossed products have became one of most
important research subjects for more than a decade. One of the most important results about the classification
of the minimal Cantor dynamical system was obtained by T. Giordano, I.F. Putnam and F. Skau in \cite{GPS}.
They proved that two minimal Cantor crossed products $C(X_1)\times_{\alpha_1}\Z$ and $C(X_2)\times_{\alpha_2}\Z$
are $*$--isomorphic iff $(X_1,\alpha_1)$ and $(X_2,\alpha_2)$ are strong orbit equivalent
(cf. \cite[Theorem 2.1]{GPS}). Recently, H. Lin and H. Matui used their notation so--called approximate
conjugacy of dynamical systems to obtain many equivalent conditions that make
$C(X_1)\times_{\alpha_1}\Z\cong C(X_2)\times_{\alpha_2}\Z$ in \cite{LM}.

For the type of dynamical systems such as $(C(X),\alpha,\Z_n)$ where $X$ is a compact Hausdorff
space and $\alpha\colon X\rightarrow X$ is homeomorphism, there is little known when
$C(X_1)\times_{\alpha_1}\Z_n\cong C(X_2)\times_{\alpha_2}\Z_n$. We only know if $\alpha_1$
acts on $X_1$ freely and $\text{H}^2(X_1/\alpha_1,\Z)$ has no element annihilated by $n$, then
$C(X_1)\times_{\alpha_1}\Z\cong C(X_2)\times_{\alpha_2}\Z$ iff $\alpha_2$ acts on $X_2$ freely and $X_1/\alpha_1\cong
X_2/\alpha_2$ (cf. \cite[Proposition 5]{LL}).

But there is a special example showed by Elliott in \cite{El} which could enlighten one on
the classification of $C(X)\times_\alpha\Z_n$ when $\alpha$ has fixed points. The example is:
$
C(\mathbf S^1)\times_\alpha\Z_2\cong\{f\colon [0,1]\rightarrow\M_2(\C)\ \text{continuous}\,
\vert\ f(0),\, f(1)\ \text{are diagonal}\},
$
where $\alpha(z)=\overline{z}$, $\forall\,z\in\mathbf S^1$ (cf. \cite[6.10.4, 10.11.5]{Bl}).

Inspired by this example, we try to find when the crossed product $C(X)\times_\alpha\Z_p$ has
the similar form and try to classify it. So, in this paper, we will consider the classification
of the crossed product $C(M)\times_\theta\mathbb Z_p$, here $M$ is a separable compact Hausdorff
space and $\theta$ is a self--homeomorphism of $M$ with prime period $p$ ($p\ge 2$) such that
$M_\theta=\{x\in M\vert\,\theta(x)=x\}\not=\varnothing$. We let $(M,\theta)$ denote the pair
which satisfy conditions mentioned above throughout the paper.

Based on the Extension theory of $C^*$--algebras and author's previous work on
$C(M)\times_\theta\Z_p$, we obtain that if $\dim M\le 2$, $M_0=M\backslash M_\theta$ is dense in
$M$ and $\mathrm H^2(M_0/\theta,\Z)\cong 0$, $\mathrm H^2(\chi(M_0/\theta),\Z)\cong 0$, then
$$
C(M)\times_\theta\Z_p\cong\{(a_{ij})_{p\times p}\in\M_p(C(M/\theta))\vert\,a_{ij}(x)=0,\
\forall\,x\in M_\theta,\ i\not=j\}.
$$
where $M/\theta$ (or $M_0/\theta$) is the orbit space of $\theta$ and $\chi(M_0/\theta)$
is the corona set of $M_0/\theta$. Moreover, let $(M',\theta')$ be another pair with
$\dim M'\le 2$ and $\overline{M\backslash M'_{\theta'}}=M'$. Then
$$
C(M)\times_\theta\Z_p\cong C(M')\times_{\theta'}\Z_p
$$
iff there is a homeomorphism $F$ from $M/\theta$ onto $M'/\theta'$ such that
$F(M_\theta)=M'_{\theta'}$.

\section{Preliminaries}

\quad\ Let $\A$ be a $C^*$--algebra with unit $1$. We denote by $\U(\A)$
the group of unitary elements of $\A$ and $\U_0(\A)$ the connected component of
the unit $1$ in $\U(\A)$. Let $\M_m(\A)$ be the matrix algebra of order $m$ over $\A$. Set (cf. \cite{R1,R2} or \cite{X})
\begin{align*}
\text{S}_n(\A)=&\{\,(a_1,\cdots,a_n)\vert\,\sum\limits^n_{k=1}a^*_ka_k=1\}\\
\csr(\A)=&\min\{\,n\vert\ \U_0(\M_n(\A))\ \text{acts transitively on}\ \text{S}_m(\A)\ \forall\, m\ge n\,\}
\end{align*}
\begin{lemma}\label{Lx}
Let $X$ be a compact Hausdorff space with covering dimension $\dim X\le 2$. Given $U\in\U(\M_p(C(X)))$,
there are $U_0\in\U_0(\M_p(C(X)))$ and $v\in\U((C(X)))$ such that $U=U_0\,\diag(1_{p-1},v)$.
\end{lemma}
\begin{proof}We have from \cite{N},
$\csr(C(X))\le\bigg[\dfrac{\dim X+1}{2}\bigg]+1\le 2$, where $[x]$ stands for the greatest integer which is less than or
equal to $x$.

Let $u_1$ be the first column of $U$. Then $u_1\in\text{S}_p(C(X))$ and consequently, there is
$U_1\in\U_0(\M_p(C(X))$ such that $u_1$ becomes the first column of $U_1$ for
 $\csr(C(X))\le 2\le p$. Put $W=U_1^*U$. Then the unitary element $W$ has the form
$W=\diag(1,V_1)$ for some $V_1\in\U(\M_{p-1}(C(X)))$.

Repeating above procedure to $V_1$, there are $U_2\in\U_0(\M_{p-1}(C(X)))$ and $V_2\in\U(\M_{p-2}(C(X)))$ such that
$V_1=U_2\,\diag(1,V_2)$. So $U=U_1U_2\,\diag(1_2,V_2)$. By induction, we can finally find $U_0\in\U_0(\M_p(C(X)))$
and $v\in\U(C(X))$ such that $U=U_0\,\diag(1_{p-1},v)$.
\end{proof}

According to \cite{Xue}, $\{x\in M\vert\, \theta^k(x)=x\}=M_\theta$,
$k=2,\cdots,p-1$. Now set $M_0=M\backslash M_\theta$. Let
$M\slash\theta$ (resp. $M_0\slash\theta$)) denote the orbit space of
$\theta$ and let $P$ be the canonical projective map of $M$ (or
$M_0$) onto $M\slash\theta$ (or $M_0\slash\theta$). Let $O_\theta(x)=\{x,\theta(x),\cdots,\theta^{p-1}(x)\}$
denote the orbit of $x$ in $M$ or $M_0$. For the pair $(M,\theta)$, the dynamical system
$(C(M),\theta,\mathbb Z_p)$ (resp. $(C_0(M_0),\theta,\mathbb Z_p))$
yields a crossed product $C^*$--algebra $C(M)\times_\theta\mathbb
Z_p$ (resp. $C_0(M_0)\times_\theta\mathbb Z_p)$. Set
\begin{align*}
\D(M,\theta)&=\left\{\begin{pmatrix} f_0&f_1&\hdots&f_{p-1}\\
                              \theta (f_{p-1})&\theta (f_0)&\hdots&\theta (f_{p-2})\\
                              \hdotsfor 4\\
                              \theta^{p-1}(f_1)&\theta^{p-1}(f_2)&\hdots&
                              \theta^{p-1}(f_0)\end{pmatrix};\, f_0,\cdots,f_{p-1}
                              \in C(M)\right\}\\
\D(M_0,\theta)&=\left\{\begin{pmatrix} f_0&f_1&\hdots&f_{p-1}\\
                              \theta (f_{p-1})&\theta (f_0)&\hdots&\theta (f_{p-2})\\
                              \hdotsfor 4\\
                              \theta^{p-1}(f_1)&\theta^{p-1}(f_2)&\hdots&
                              \theta^{p-1}(f_0)\end{pmatrix};\, f_0,\cdots,f_{p-1}
                              \in C_0(M_0)\right\},
\end{align*}
where $\theta (f)(x)=f(\theta (x)),\forall x\in M$ (resp. $M_0$), $f\in C(M)$ (resp. $C_0(M_0)$).
By 7.6.1 and 7.6.5 of \cite{Pd}, we have $C(M)\times_\theta\mathbb Z_p\cong\D(M,\theta)\subset\text{M}_n(C(M))$ and
$C_0(M_0)\times_\theta\mathbb Z_p\cong\D(M_0,\theta)\subset\text{M}_n(C_0(M_0))$.

 Let $\omega = \text{e}^{{2\pi i} \slash p}$. Put
 $e_j=\diag(\underset{j-1}{\underbrace{0,\cdots,0}},1,\underset{p-j}{\underbrace{0,\cdots,0}})$ and
$$
\Omega_p=\begin{pmatrix} {1\over\sqrt p}&{\omega\over\sqrt p}&\hdots&{\omega^{p-1}
                      \over\sqrt p}\\
                    {1\over\sqrt p}&{\omega^2\over\sqrt p}&\hdots&{(\omega^2)^{p-1}
                      \over\sqrt p}\\
                    \hdotsfor 4\\
                    {1\over\sqrt p}&{\omega^{p-1}\over\sqrt p}&\hdots&{(\omega^{p-1})^{p-1}
                      \over\sqrt p}\\
                    {1\over\sqrt p}&{1\over\sqrt p}&\hdots&{1\over\sqrt p}
\end{pmatrix},\quad P_j=\Omega_p^*e_j\Omega_p,\ j=1,\cdots,p.
$$
It is easy to check that $\Omega_p$ is unitary in
$\M_p(\C)$ and $P_1,\cdots,P_p$ are projections in $\D(M,\theta)$. Define $*$--homomorphism
$\pi\colon\D(M,\theta)\rightarrow\doplus^{p-1}_{j=0}C(M_\theta)$ as
$$
\pi\left(\!\begin{pmatrix}f_0\!&f_1\!&\hdots\!&f_{p-1}\!\\
                      \theta (f_{p-1})&\theta (f_0)&\hdots&\theta (f_{p-2})\\
                    \hdotsfor 4\\
                  \theta^{p-1}(f_1)&\theta^{p-1}(f_2)&\hdots&\theta^{p-1}(f_0)
\end{pmatrix}\!\right)\!
=\!\left(\sum^{p-1}_{j=0}\omega^{j(p-1)}f_j\vert_{M_\theta},\cdots,\!
                            \sum^{p-1}_{j=0}f_j\vert_{M_\theta}\right ).
$$
$\pi$ induces following exact sequence of $C^*$--algebras:
\begin{equation}\label{XX}
0\longrightarrow\D(M_0,\theta)\stackrel{l}{\longrightarrow}\D(M,\theta)\stackrel{\pi}{\longrightarrow}
\doplus^{p-1}_{j=0}C(M_\theta)\longrightarrow 0.
\end{equation}

\begin{lemma}\label{L1}
For the pair $(M,\theta)$, we have
\begin{enumerate}
\item[(1)] Every irreducible representation of $\D(M_0,\theta)$ is equivalent to the
representation $\pi_x$, where $\pi_x(a)=a(x)$ for some $x\in M_0$
and $\forall a\in\D(M_0,\theta)$ and $P(x)\rightarrow [\pi_x]$ gives
a homeomorphism of $M_0\slash\theta$ onto
$\widehat{\D(M_0,\theta)}$--the spectrum of $\D(M_0,\theta)$, where
we identify $\D(M_0,\theta)$ with $\{\,a\in\D(M,\theta)\vert\,a\vert_{M_\theta}=0\};$
\item[(2)] $\D(M_0,\theta)$ is a $p$--homogeneous algebra which is $*$--isomorphic to
$C_0(M_0\slash\theta, E)$, where $E$ is a matrix  bundle over $M_0\slash\theta$ with fiber $\M_p(\C);$
\item[(3)] Let $\sigma$ be a pure state on $\D(M,\theta)$. Then $\sigma$ is multiplicable iff there is $x_0\in M_\theta$
such that
$$
\sigma\left(\begin{pmatrix} f_0&f_1&\hdots&f_{p-1}\\
                              \theta (f_{p-1})&\theta (f_0)&\hdots&\theta (f_{p-2})\\
                              \hdotsfor 4\\
                              \theta^{p-1}(f_1)&\theta^{p-1}(f_2)&\hdots&
                              \theta^{p-1}(f_0)\end{pmatrix}\right)=\sum\limits^{p-1}_{j=0}\omega^{jk}f_j(x_0)
$$
for some $k\in\{0,\cdots,p-1\}$. We let $\sigma_{x_0,k}$ denote the $\sigma$.
\end{enumerate}
\end{lemma}
\begin{proof}
Let $(\Pi,K)$ be an irreducible representation of $\D(M_0,\theta)$. Then by \cite[Proposition 2.10.2]{Di},
there is an irreducible representation $(\rho,H)$ of $\text{M}_p(C_0(M_0))$ such that $K\subset H$ and $\Pi(\cdot)=
\rho(\cdot)\vert_K$. It is well-known that $(\rho,H)$ is equivalent to $(\rho_{x_0},\mathbb C^p)$ for some $x_0\in
M_0$, where $\rho_x((f_{ij})_{p\times p})=(f_{ij}(x))_{p\times p}$, $x\in M_0$ and $f_{ij}\in C_0(M_0)$, $i,j=1,\cdots,p$.

Put $\pi_x=\rho_x\vert_{\D(M_0,\theta)}$, $x\in M_0$. Given $A=(a_{ij})_{p\times p}\in
\text{M}_p(\mathbb C)$. Since $x_0,\theta(x_0),\cdots,$ $\theta^{p-1}(x_0)$ are mutually
different in $M_0$, we can find $f_0,\cdots,f_{p-1}\in C(M)$ such that
$f_j\vert_{M_\theta}=0,\ j=0,\cdots,p-1$ and
$$
\pi_{x_0}\left(\begin{pmatrix} f_0&f_1&\hdots&f_{p-1}\\
                              \theta (f_{p-1})&\theta (f_0)&\hdots&\theta (f_{p-2})\\
                              \hdotsfor 4\\
                              \theta^{p-1}(f_1)&\theta^{p-1}(f_2)&\hdots&
                              \theta^{p-1}(f_0)\end{pmatrix}\right)=A.
$$
This means that $(\pi_{x_0},\mathbb C^p)$ is an irreducible representation of $\D(M_0,\theta)$ and hence $(\Pi,K)$ is
equivalent to $(\pi_{x_0},\mathbb C^p)$. Thus $\D(M_0,\theta)$ is a p--homogeneous $C^*$--algebra.

Let $x_0\in M_0$ and $x_1=\theta^j(x_0)$ for some $j\in\{1,\cdots,p-1\}$. It is easy to check that there is a unitary
matrix $U$ in $\text{M}_p(\mathbb C)$ such that $\pi_{x_0}(a)=U\pi_{x_1}(a)\,U^*$, $\forall\,a\in\D(M_0,\theta)$, i.e.,
$(\pi_{x_0},\mathbb C^p)$ and $(\pi_{x_1},\mathbb C^p)$ are equivalent; On the other hand, if $(\pi_{x_0},\mathbb C^p)$
and $(\pi_{x_1},\mathbb C^p)$ are equivalent and $O_\theta(x_0)\cap O_\theta(x_1)=\varnothing$, then we can choose
$g_0,\cdots,g_{p-1}\in C(M)$ such that $g_j(x)=0,\ \forall\,x\in M_\theta\cup O_\theta(x_0)$ and
$g_j\vert_{O_\theta(x_1)}\not=0$, $j=0,\cdots,p-1$. Set
$$
G=\begin{pmatrix} g_0&g_1&\hdots&g_{p-1}\\
                              \theta (g_{p-1})&\theta (g_0)&\hdots&\theta (g_{p-2})\\
                              \hdotsfor 4\\
                              \theta^{p-1}(g_1)&\theta^{p-1}(g_2)&\hdots&
                              \theta^{p-1}(g_0)\end{pmatrix}\in\D(M_0,\theta).
$$
Then $\pi_{x_0}(G)=0$, while $\pi_{x_1}(G)\not=0$. Therefore, we have $O_\theta(x_0)\cap O_\theta(x_1)\not=\varnothing$,
that is, $x_0=\theta^j(x_1)$ for some $j\in\{0,1,\cdots,p-1\}$. So the map $P(x)\mapsto [\pi_x]$ gives a homeomorphism
of $M_0/\theta$ onto $\widehat{\D(M_0,\theta)}$ by using \cite[Lemma 3.3.3]{Di}.

Set $[x]=P(x)\in M_0/\theta$ and $D([x])=\D(M_0,\theta)/\Ker\pi_x\cong\text{M}_p(\mathbb C)$, $\forall\,x\in M_0$.
Let $a([x])$ be the canonical image of $a\in\D(M_0,\theta)$ in $D([x])$. Put $E=\dcup_{[x]\in M_0/\theta}D([x])$.
Then $E$ is a fiber bundle (matrix bundle) over $M_0/\theta$ with fiber $\text{M}_p(\mathbb C)$ (cf. \cite[\S 3.2, P.249]{Fell}).
The $*$--isomorphism from $\D(M_0,\theta)$ onto $C_0(M_0/\theta,E)$ is defined by $\phi(a)([x])=a([x])$, $\forall\,
a\in\D(M_0,\theta)$ and $[x]\in M_0/\theta$.

The proof of (3) comes from Corollary 1, Case 1 and case 2 on P77 and case 2 on P79 of \cite{LL}.
\end{proof}

Assume that the matrix bundle $E$ above is trivial, i.e., there is a homeomorphism $\Gamma\colon E\rightarrow
M_0/\theta\times\M_p(\C)$ such that $\Gamma(D([x]))=[x]\times\M_p(\C)$, $\forall\,[x]\in M_0/\theta$. Then $\Gamma$
induces an isomorphism $\Gamma^*\colon C_0(M_0/\theta,E)\rightarrow \M_p(C_0(M_0/\theta))$ given by $\Gamma^*(f)
([x])=\Gamma(f([x]))$, $\forall\,f\in C_0(M_0/\theta,E)$ and $[x]\in M_0/\theta$. Set
$\Phi=l\circ\phi^{-1}\circ\Gamma^{-1}$. Applying Lemma \ref{L1} to (\ref{XX}), we have following:
\begin{corollary}\label{CA}
Let $(M,\theta)$ be the pair such that all matrix bundles over $M_0/\theta$ is trivial. Then we
have following exact sequence of $C^*$--algebras
\begin{equation}\label{YY}
0\longrightarrow\M_p(C_0(M_0/\theta))\stackrel{\Phi}{\longrightarrow}\D(M,\theta)
\stackrel{\pi}{\longrightarrow}\doplus^{p-1}_{j=0} C(M_\theta)\longrightarrow 0.
\end{equation}
\end{corollary}

\section{Main results}
\setcounter{equation}{0}
\quad\ Consider the exact sequence of $C^*$--algebras
\begin{equation}\label{2X}
0\longrightarrow\B\stackrel{i}{\longrightarrow}\E\stackrel{q}{\longrightarrow}\A\longrightarrow 0,
\end{equation}
where $\B$ is a separable $C^*$--algebra and $i(\B)$ is the essential ideal of $\E$. Let $\{u_n\}$ be an approximate identity of
$\B$ and $M(\B)$ be the multiplier algebra of $\B$. Define $*$--monomorphism $\rho\colon\E
\rightarrow M(\B)$ by $\mu(e)=\lim\limits_n i^{-1}(i(u_n)e)$, $\forall\,e\in\E$, where
the limit is taken as the strict limit in $M(\B)$. Then Busby invariant
$\tau_q\colon\A\rightarrow M(\B)/\B$, associated with (\ref{2X}) is given by
$\tau_q(a)=\Omega(\mu(e))$, where $e\in\E$ such that $q(e)=a$,
$\Omega\colon M(\B)\rightarrow M(\B)/\B$ is the quotient map. $\tau_q$ is well--defined and
is monomorphic since $i(\B)$ is essential. Please see \cite[Chapter 3]{JT} or
\cite[Chapter 5]{Lin}, or \cite[Chapter 3]{WO}) for details. In order to our main result,
we need following lemma, which comes from \cite[Lemma 3.2.2]{JT}.
\begin{lemma}\label{LYY}
Let $\tau_q$ (resp. $\tau_{q'}$) be the Busby invariant associated following exact sequence of
$C^*$--algebras:
$$
0\longrightarrow\B\stackrel{i}{\longrightarrow}\E\stackrel{q}{\longrightarrow}\A\longrightarrow 0\quad
(\text{resp.}\ 0\longrightarrow\B\stackrel{i'}{\longrightarrow}\E'\stackrel{q'}{\longrightarrow}
\A\longrightarrow 0),
$$
where $i(\B)$ (resp. $i'(\B)$) is the essential ideal of $\E$ (resp. $\E'$). Assume that there is a unitary
element $u\in M(\B)$ such that $\tau_q(a)=\Omega(u)\tau_{q'}(a)\Omega(u^*)$, $\forall\,a\in\A$.
Then $\E$ is $*$--isomorphic to $\E'$.
\end{lemma}

Let $X$ be a locally compact Hausdorff space. Let $X^+$ (resp. $\beta(X)$) denote the
one--point (resp. Stone--C\v{e}ch) compactification of $X$. Set $\chi(X)=\beta(X)\backslash X$ (the corona set of
$X$).
\begin{lemma}\label{LY}
Let $X$ be a locally compact Hausdorff space with $\dim X\le 2$.
If ${\mathrm H}^2(X,\Z)\cong 0$, then every matrix bundle over $X$ is trivial.
\end{lemma}
\begin{proof} Let $M(X^+)$ denote the collection of all
(isomorphism classes of) matrix bundles of arbitrary degree over $X^+$ and $SM(X^+)$ denote the the distinguished
subsemigroup of $M(X)$, consisting of matrix bundles of the form $V\otimes V^*$,  where $V$ is an arbitrary complex
vector bundle over $X^+$, and $*$ denotes adjoints.
When $\dim X^+=\dim X\le 2$ and $\text{H}^2(X,\Z)=\mathrm H^2(X^+,\Z)\cong 0$, $M(X^+)=SM(X^+)$ by
\cite[Corollary 1]{RH} and all complex vector bundles over $X^+$ are trivial. Therefore, all matrix bundles of arbitrary
degree over $X^+$ is trivial. Because every matrix bundle over $X$ with fiber $\M_n(\C)$ can be extended to an matrix
bundle over $X^+$ with fiber $\M_n(\C)$ (using the same methods described in \cite[P110--P112]{Hus}), we get the assertion.
\end{proof}

By Lemma \ref{LY}, we may assume that the pair $(M,\theta)$ satisfies condition (A):
$$
\dim M\le 2\quad\text{and}\quad\mathrm H^2(M_0/\theta,\Z)\cong 0.
$$
\begin{lemma}\label{LZ}
Let the pair $(M,\theta)$ satisfy Condition (A). If $M_0=M\backslash M_\theta$ is dense in $M$,
then $\Phi(\M_p(C_0(M_0/\theta))$ is an essential ideal of $\D(M,\theta)$.
\end{lemma}
\begin{proof} Let $a\in\D(M,\theta)$ such that $ac=0,\ \forall\,c\in\D(M_0,\theta)=
\Phi(\M_p(C_0(M_0/\theta))$. If there is $x_0\in M_0$ such that $a(x_0)\not=0$, then we can find
a $b\in\D(M_0,\theta)$ such that $b(x_0)=(a(x_0))^*$. It follows that $a(x_0)
(a(x_0))^*=0$, i.e., $a(x_0)=0$, a contradiction. So $a\vert_{M_0}=0$. Since $M_0$ is dense in $M$ and all elements
in $a$ are continuous, we have $a=0$.
\end{proof}

It is known that
\begin{align*}
M(\M_p(C_0(M_0/\theta)))=\M_p(C_b(M_0/\theta))&\cong \M_p(C(\beta(M_0/\theta))),\\
\M_p(C_b(M_0/\theta))\slash \M_p(C_0(M_0/\theta))&\cong\M_p(C(\chi(M_0/\theta))).
\end{align*}

Assume that $(M,\theta)$ satisfies Condition (A) and the condition
that $M\backslash M_\theta$ is dense in $M$. Then the $*$--monomorphism
$\mu\colon\D(M,\theta)\rightarrow \M_p(C_b(M_0/\theta))$ is given by
$\mu(a)=\lim\limits_{n\to\infty}\Phi^{-1}(\Phi(u_n)a)$, $\forall\,a\in\D(M,\theta)$, where
$\{u_n\}$ is an approximate identity for $\M_p(C_0(M_0/\theta))$.
We now construct the Busby invariant $\tau_{\pi}$ associated with (\ref{YY}) as follows.

Let $f_0,\cdots,f_{p-1}\in C(M_\theta)$. Extends them to $g_0,\cdots,g_{p-1}\in C(M)$, respectively, such that
$g_j\vert_{M_\theta}=f_j$, $j=0,\cdots,p-1$. Put $\hat{f}_j=\dfrac{1}{\,p\,}\sum\limits^{p-1}_{k=0}\theta^k(g_j)$,
$j=0,\cdots,p-1$. Then $\theta(\hat{f}_j)=\hat{f}_j$ and $\hat{f}_j\vert_{M_\theta}=f_j$, $j=0,\cdots,p-1$. Thus,
functions $h_0,\cdots,h_{p-1}$ given by $h_j([x])=\hat{f}_j(x)$, $x\in M_0$, $j=0,\cdots,p-1$ are all in $C_b(M_0/\theta)$.
Set $a(f_0,\cdots,f_{p-1})=\sum\limits^{p-1}_{j=0}P_{j+1}\hat{f}_j$. Then $\pi(a(f_0,\cdots,f_{p-1}))=
(f_0,\cdots,f_{p-1})\in\doplus^{p-1}_{j=0}C(M_\theta)$.

Let $\Omega\colon \M_p(C_b(M_0/\theta))\rightarrow \M_p(C_b(M_0/\theta))\slash \M_p(C_0(M_0/\theta))$ be the
canonical homomorphism. Note that for any $b\in \M_p(C_0(M_0/\theta))$ and $[x]\in M_0/\theta$,
\begin{align*}
(\mu(P_{j+1}\hat{f}_j)b)([x])=&\Phi^{-1}(P_{j+1}\hat{f}_j\Phi(b))([x])
=\Gamma^*\circ\phi(P_{j+1}\hat{f}_j\Phi(b))([x])\\
=&\Gamma((P_{j+1}\hat{f}_j\Phi(b))([x])).
\end{align*}
since $P_{j+1}\hat{f}_j\Phi(b)-P_{j+1}\Phi(b)\hat{f}_j(x)\in\Ker\pi_x$, it follows that
$$
\Gamma((P_{j+1}\hat{f}_j\Phi(b))([x]))=\Gamma((P_{j+1}\Phi(b))([x]))\hat{f}_j(x)=(\mu(P_{j+1})b)([x])h_j([x]),
$$
that is, $\mu(P_{j+1}\hat{f}_j)=\mu(P_{j+1})h_j$, $j=0,\cdots,p-1$.
Set $Q_j=\Omega\circ\mu(P_j)$, $j=1,\cdots,p$. Then $\tau_\pi$ is given by
$$
\tau_\pi(f_0,\cdots,f_{p-1})=\Omega\circ\mu(a(f_0,\cdots,f_{p-1}))=
\sum\limits^{p-1}_{j=0}Q_{j+1}\Omega(h_j1_p).
$$

Let $\A(M_\theta)=\{(a_{ij})_{p\times p}\in\M_p(C(M/\theta))\vert\,a_{ij}(x)=0,\,x\in M_\theta,\, i\not=j\}$.
Define the $*$--homomorphism $\Lambda\colon\A(M_\theta)\rightarrow\doplus^{p-1}_{j=0}C(M_\theta)$
by $\Lambda((a_{ij})_{p\times p})=(a_{11}\vert_{M_\theta},\cdots,a_{pp}\vert_{M_\theta})$.
Then we have following exact sequence:
\begin{equation}\label{ZZ}
0\longrightarrow\M_p(C_0(M_0/\theta))\stackrel{i}{\longrightarrow}\A(M_\theta)
\stackrel{\Lambda}{\longrightarrow}\doplus^{p-1}_{j=0}C(M_\theta)\longrightarrow 0.
\end{equation}
The Busby invariant $\tau_\Lambda$ associated with (\ref{ZZ}) is given by
$$
\tau_\Lambda(f_0,\cdots,f_{p-1})=\sum\limits^{p-1}_{j=0}\Omega(e_{j+1})\Omega(h_j1_p),
$$
where $h_0,\cdots,h_{p-1}$ are given as above.

Now we present the main results of the paper as follows
\begin{theorem}\label{PX}
Suppose that the pair $(M,\theta)$ satisfies condition (A) and Condition (B):
$$
M_0=M\backslash M_\theta\ \text{is dense in}\ M,\quad \mathrm{H}^2(\chi(M_0/\theta),\Z)\cong 0.
$$
Then $C(M)\times_\theta\Z_p\cong\A(M_\theta)$.
\end{theorem}
\begin{proof}
Since $\dim(M_0/\theta)=\dim M_0\le \dim M\le 2$
by \cite[Lemma 1.3]{Xue} and $\mathrm H^2(M_0/\theta,\Z)$ $\cong 0$, it follows from Lemma \ref{LY}
that every matrix bundle over $M_0/\theta$ is trivial. Since
$$
\dim(\chi(M_0/\theta))\le\dim(\beta(M_0/\theta))\le\dim(M_0/\theta))\le 2,\quad
\text{H}^2(\chi(M_0/\theta),\Z)\cong 0,
$$
all complex vector bundles over $\chi(M_0/\theta)$ are trivial.

It is easy to check that for any $b\in\M_p(C_0(M_0/\theta))$, there is $f_j\in C_0(M_0)$ with
$\theta(f_j)=f_j$ such that $P_j\Phi(b)P_j=f_jP_j$, $j=1,\cdots,p$. So,
$\mu(P_j)\M_p(C_0(M_0/\theta))\mu(P_j)$ is a commutative $C^*$--algebra, $j=1,\cdots,p$.
Since $\M_p(C_0(M_0/\theta))$ is dense in $\M_p(C_b(M_0/\theta))$ in the sense of strict
topology, $\mu(P_j)\M_p(C_b(M_0/\theta))\mu(P_j)$ is also a commutative $C^*$--algebra,
$j=1,\cdots,p$. Assume that $Q_1,\cdots,Q_p$ in $\M_p(C(\chi(M_0/\theta))$. Then
$Q_j\M_p(C(\chi(M_0/\theta)))\, Q_j$ is commutative and hence $\rank Q_j(x)\le 1$, $\forall\,
x\in\chi(M_0/\theta)$ by \cite[Lemma 6.1.3]{Pd}, $j=1,\cdots,p$.

Note that $Q_1,\cdots,Q_p$ are mutually orthogonal and $\sum\limits^p_{j=1}Q_j=1_p$. Therefore,
we have $\rank Q_j(x)=1,\ \forall\,x\in\chi(M_0/\theta)$, $j=1,\cdots,p$. So there are partial
isometries $V_1,\cdots,V_p$ in $\M_p(C(\chi(M_0/\theta))$ such that $S_j=V^*_jV_j$ and
$e_j=V_jV^*_j$, $j=1,\cdots,p$. Put $V=\sum\limits^p_{j=1}V_j$. Then $V$ is a unitary element
in $\M_p(C(\chi(M_0/\theta))$ and $S_j=V^*e_jV$, $j=1,\cdots,p$. Applying Lemma \ref{Lx} to $V^*$,
we can find $V_0\in\U_0(\M_p(C(\chi(M_0/\theta)))$ and $v\in \U(C(\chi(M_0/\theta)))$ such that
$V^*=V_0\,\diag(1_{p-1},v)$. Consequently, $S_j=V_0e_jV^*_0$, $j=1,\cdots,p$.

Similarly, there is $V'_0\in\U_0(\M_p(C(\chi(M_0/\theta))))$ such that $\Omega(e_j)=V'_0e_jV_0'^*$.
Choose $U\in\U_0(\M_p(C(\beta(M_0/\theta))))$ such that $\Omega(U)=V_0V_0'^*$. Since
$\Omega(h_j1_p)$ commutes with every element in $\M_p(C(\chi(M_0/\theta))), j=0,\cdots,p-1$,
it follows that
$$
\tau_\pi(f_0,\cdots,f_{p-1})=\Omega(U)\tau_\Lambda(f_0,\cdots,f_{p-1})\Omega(U^*),
$$
$\forall\,f_0,\cdots,f_{p-1}\in C(M_\theta)$. Therefore, $C(M)\times_\theta\Z_p\cong\A(M_\theta)$
by Lemma \ref{LYY}.
\end{proof}

For the pair $(M,\theta)$ with $\dim M\le 1$, we have $\dim(M_0/\theta)\le 1$ and
$\dim\chi(M_0/\theta)\le 1$. In this case, $\mathrm H^2(M_0/\theta,\Z)\cong\mathrm
H^2(\chi(M_0/\theta),\Z)\cong 0$. Therefore, we have following corollary according to Theorem
\ref{PX}:
\begin{corollary}
Let $(M,\theta)$ be the pair with $\dim M\le 1$ and $\overline{M_0}=M$, where $\overline{M_0}$
is the closure of $M_0$ in $M$. Then
$C(M)\times_\theta\Z_p\cong\A(M_\theta)$.
\end{corollary}
\begin{theorem}\label{T}
Suppose that the pair $(M,\theta)$ satisfy Condition (A), Condition (B).
Let $(M',\theta')$ be another pair with $\overline{M'_0}=M'$. Then
$C(M)\times_\theta\Z_p\cong C(M')\times_{\theta'}\Z_p$ iff there exists a homeomorphism
$F\colon M/\theta\rightarrow M'/\theta'$ such that $F(M_\theta)=M'_{\theta'}$.
\end{theorem}
\begin{proof}
($\Leftarrow$) Put $\alpha=F\vert_{M_0/\theta}$. Then $\alpha$ has a unique homeomorphic extension
$\bar\alpha\colon\beta(M_0/\theta)\rightarrow\beta(M'_0/{\theta'})$
(cf. \cite[\S 44 Corollary 10]{H}). Thus $\bar\alpha\colon\chi(M_0/\theta)\rightarrow
\chi(M'_0/\theta')$ is a homeomorphism. Thus $(M',\theta')$ Satisfies Condition (A) and
Condition (B). Consequently, $C(M)\times_\theta\Z_p\cong\A(M_\theta)$ and
$C(M')\times_{\theta'}\Z_p\cong\A(M'_{\theta'})$ by Theorem \ref{PX}. Clearly,
$\Psi((a_{ij})_{p\times p})=(a_{ij}\circ F)_{p\times p}$ gives a $*$--isomorphism from
$\A(M'_{\theta'})$ onto $\A(M_\theta)$. The assertion follows.

($\Rightarrow$) Let $\Delta$ be the $*$--isomorphism from $\D(M',\theta')$ onto
$\D(M,\theta)$. Let $a'\in\D(M'_0,\theta')$ and put $a=\Delta(a')$. If $a\not\in\D(M_0,\theta)$,
we can pick $y_0\in M_{\theta}$ such that $a(y_0)\not=0$. Then $\sigma_{y_0,1}\circ\Delta$ is
multiplicable on $\D(M',\theta')$. Thus, there exist $x_0 \in M'_{\theta'}$ and
$k\in\{1,\cdots,p\}$ such that $\sigma_{y_0,1}\circ\Delta=\sigma_{x_0,k}$ and hence
$\sigma_{y_0,1}(a)=\sigma_{x_0,k}(a')=0$, a contradiction. So, $\Delta$ induces a
$*$--isomorphism $\Delta_0\colon\D(M'_0,\theta')\rightarrow\D(M_0,\theta)$ and so that
$\Delta_0'=\Phi^{-1}\circ\Delta_0\circ\Phi'^{-1}$ gives a $*$--isomorphism of
$\M_p(C_0(M'_0/\theta'))$ onto $\M_p(C_0(M_0/\theta))$. Thus, we can find a homeomorphism
$\alpha\colon M_0/\theta\rightarrow M'_0/\theta'$. This shows that $(M',\theta')$ satisfies
Condition (A) and Condition (B) too.

Now we have $\A(M'_{\theta'})\cong\A(M_\theta)$ via the $*$--isomorphism $\Theta$ by Theorem
\ref{PX}. Since $\Theta(f1_p)$ commutes with every element in $\A(M_{\theta})$, $\forall\, f\in
C(M'/\theta')$, it follows that
there is $h\in C(M/\theta)$ such that $\Theta(f1_p)=h1_p$. Thus, $f\mapsto h$ yields a $*$--isomorphism
from $C(M'/\theta')$ onto $C(M/\theta)$ so that there is a homeomorphism $F\colon M/\theta\rightarrow
M'/\theta'$ such that $\Theta(f1_p)=(f\circ F)1_p=h1_p$.

Let $\phi$ be a character on $\A(M_\theta)$ with $\phi(1_p)=1$. Since $e_1,\cdots,e_p
\in\A(M_\theta)$ and $\sum\limits^p_{j=1}e_j=1_p$, there is $e_{i_0}$ such that $\phi(e_{i_0})=1$
and $\phi(e_j)=0$, $j\not=i_0$. Without losing the generality, we may assume $i_0=1$. Then
$f\mapsto\phi(fe_1)$ is a character on $C(M/\theta)$. Thus, there is $x_0\in M/\theta$ such that
$\phi(fe_1)=f(x_0)$, $\forall\,f\in C(M/\theta)$. For any $g\in C_0(M_0/\theta)$, let
$B=(b_{ij})_{p\times p}\in\A(M_\theta)$ be given by $b_{p1}=g$ and $b_{ij}=0$, $i\not=p$,
$j\not=1$. Then $B^*B=g^*g\,e_1$, $e_pB=B$ and hence $|g(x_0)|^2=0$,
$\forall\,g\in C_0(M_0/\theta)$. Consequently, $x_0\in M_\theta$.

For any $x\in M_\theta$, define the character $\phi_x$ on $\A(M_\theta)$ by
$\phi_x((a_{ij})_{p\times p})=a_{11}(x)$. Note that $\phi_x\circ\Theta^{-1}$ is a character on
$\A(M'_{\theta'})$. So by above arguments, there is $y\in M'_{\theta'}$ such that
$$
g(y)=\phi_x\circ\Theta(g1_p)=\phi_x((g\circ F)1_p)=g(F(x)),\quad\forall\,g\in C(M'/\theta').
$$
This means that $F(M_\theta)\subset M'_{\theta'}$. Similarly, we have $F^{-1}(M'_{\theta'})
\subset M_\theta$. Thus, $F(M_\theta)=M'_{\theta'}$.
\end{proof}
\begin{corollary}
Suppose that the pair $(M,\theta)$ satisfy Condition (A) and Condition (B).
Let $(M',\theta')$ be another pair. If $(M,\theta)$ and $(M',\theta')$ are orbit equivalent,
that is, there is a homeomorphism $F\colon M\rightarrow M'$ such that $F(O_\theta(x))=O_{\theta'}(F(x)),\forall\,
x\in M$, then $C(M)\times_\theta\Z_p\cong C(M')\times_{\theta'}\Z_p$.
\end{corollary}
\begin{proof}
$F$ induces a homeomorphism $\tilde F$ of $M/\theta$ onto $M'/\theta'$ given by $\tilde F(P(x))$ $=P'(F(x))$
by the assumption, where $P'\colon M'\rightarrow M'/\theta'$ is the canonical projective map.
Obviously, $\tilde F(M_\theta)=M'_{\theta'}$ and $\overline{M'_0}=F(\overline{M_0})=M'$.
So $C(M)\times_\theta\Z_p\cong C(M')\times_{\theta'}\Z_p$ by Theorem \ref{T}.
\end{proof}

\section{Some examples}

\begin{example}
{\rm
Consider $(M,\theta)$, where $M=\mathbf S^1$ and $\theta(z)=\bar z$, $\forall\,z\in\mathbf S^1$.
Then
$$
M_\theta=\{-1,1\},\ \overline{M_0}=M,\ M_0/\theta\cong (-1,1),\ M/\theta\cong[-1,1].
$$
$(M,\theta)$ satisfies Condition (A) and (B) for $\dim M=1$. Thus $C(M)\times_\theta
\mathbb Z_p\cong\A(M_\theta)$ by Theorem \ref{PX}.

Define $\gamma(<z>)=\dfrac{x+1}{2}$ for $z=x+i\,y\in\mathbf S^1$, where $<z>=P(z)\in M/\theta$.
Clearly, $\gamma$ is a homeomorphism from $M/\theta$ onto $[0,1]$ and $\gamma(M_\theta)=\{0,1\}$.
So
\begin{align*}
C(M)\times_\theta\mathbb Z_2\cong&\A(\{0,1\})\\
   =&\{f\colon [0,1]\rightarrow\M_2(\C)\ \text{continuous}
\vert\ f(0),\, f(1)\ \text{are diagonal}\}.
\end{align*}
}
\end{example}

\begin{example}
{\rm
Let $M={\mathbf S}^1\Times{\mathbf S}^1$ and $\theta(z_1,z_2)=(z_2,z_1)$, $\forall\,
z_1,z_2\in{\mathbf S}^1$. Then $M_\theta=\{(z,z)\vert\,z\in\mathbf S^1\}\cong\mathbf S^1$ and
$\overline{M_0}=M$. Set
$$
S=\{(z_1+z_2,z_1z_2)\vert\,z_1,\,z_2\in\mathbf S^1\}\subset\C\Times\C.
$$
It is easy to check that $S$ is a closed and bounded subset in $\C\Times\C$, that is, $S$
is compact. Define the continuous map $\xi\colon M/\theta\rightarrow S$ and
$\beta\colon [0,1]\Times\mathbf S^1\rightarrow S$, respectively, by
$\xi(<\!z_1,z_2\!>)=(z_1+z_2,z_1z_2)$ and $\beta(t,z)=(2zt,z^2)$,
where $<\!z_1,z_2\!>=P(z_1,z_2)\in M/\theta$. Then $\xi$ and $\beta$ are all homeomorphic
(cf. \cite[Exmple 4.3]{Xue}). Therefore the homeomorphism $\delta=\beta^{-1}\circ\xi\colon
M/\theta\rightarrow [0,1]\Times\mathbf S^1$ sends $M_\theta$ to $\{1\}\Times\mathbf S^1$.

\noindent{\bf Claim 1.} $(M,\theta)$ satisfies Condition (A) and Condition (B).

Since $\delta(M_\theta)=\{1\}\Times\mathbf S^1$, $\delta(M_0/\theta)=[0,1)\Times
\mathbf S^1$. Let $i\colon\{1\}\Times\mathbf S^1\rightarrow[0,1]\Times\mathbf S^1$ be the
inclusion. Then $i$ is homotopic equivalence map. Since $\mathrm H^2([0,1]\Times\mathbf S^1)
\cong 0$, it follows from the exact sequence of the reduced coholomological groups (cf. \cite{Ta}) that
$$
\tilde{\mathrm H}^1([0,1]\Times\mathbf S^1,\Z)\stackrel{i^*}{\longrightarrow}
{\mathrm H}^1(\{1\}\Times\mathbf S^1,\Z)\longrightarrow
\tilde{\mathrm H}^2(([0,1)\Times\mathbf S^1)^+,\Z)\longrightarrow 0.
$$
So $\tilde{\mathrm H}^2(([0,1)\Times\mathbf S^1)^+,\Z)\cong 0$ and consequently, $(M,\theta)$ satisfies Condition (A).

Noting that $[0,1)\Times\mathbf S^1\cong [0,+\infty)\Times\mathbf S^1$, we have
$\mathrm H^2(\chi([0,1)\Times\mathbf S^1),\Z)\cong 0$ by \cite[Corollary 4.7]{GP}.
So $(M,\theta)$ satisfies Condition (B).

Let $M'=[0,2]\Times\mathbf S^1$ and $\theta'(t,z)=(2-t,z)$. Then $M'_{\theta'}=\{1\}\Times
\mathbf S^1$. Define
the homeomorphism $\delta'\colon M'/\theta'\rightarrow [0,1]\Times\mathbf S^1$ by
$\delta'(<t,z>)=(1-|1-t|,z)$.

\noindent{\bf Claim 2.} $C(M)\times_\theta\Z_2\cong C(\mathbf S^1)\otimes\A(1)$, where
$$
\A(1)=\bigg\{\begin{pmatrix}f_{11}&f_{12}\\ f_{21}&f_{22}\end{pmatrix}\in
\M_2(C([0,1]))\bigg\vert\,f_{12}(1)=f_{21}(1)=0\bigg\}.
$$

Since $M/\theta\cong [0,2]\Times\mathbf S^1/\theta'$ via $\delta'^{-1}\circ\delta$ and
$(\delta'^{-1}\circ\delta)(M_\theta)=M'_{\theta'}$, it follows from Theorem \ref{T} that
$$
C(M)\times_\theta\Z_2\cong C(M')\times_{\theta'}\Z_2\cong C(\mathbf S^1)
\otimes(C([0,2])\times_{\theta_1}\Z_2),
$$
where $\theta_1\colon [0,2]\rightarrow [0,2]$ given by $\theta_1(t)=2-t$.

Note that $[0,2]/\theta_1\cong [0,1]$ via $<t>\mapsto |1-t|$, $\forall\,t\in [0,1]$. We have
$C([0,2])\times_{\theta_1}\Z_2\cong\A(1)$ by Theorem \ref{PX}.
}
\end{example}

\bigskip

\noindent Department of Mathematics\\
East China Normal University\\
 Shanghai 200241, P.R. China

\noindent {\it E-mail}: {\tt yfxue@math.ecnu.edu.cn
\end{document}